\documentclass[12pt]{amsart}
\usepackage{amssymb,latexsym}
\usepackage{amsthm}
\usepackage{amsfonts}
\usepackage{amssymb}
\usepackage{enumerate}
\newtheorem{theorem}{Theorem}

\theoremstyle{definition}

\newtheorem{example}[theorem]{Example}
\newtheoremstyle{remark}{9pt}{9pt}{}{0pt}{\bf}{.}{0.5em}{}
\theoremstyle{remark} \newtheorem{remark}{Remark}

\def\id{\operatorname{id}}

\def\aut{\operatorname{Aut}}
\def\tr{\operatorname{tr}}
\def\diag{\operatorname{diag}}
\title{Note on the group of automorphisms of the Spectral Ball}
\author{\L ukasz Kosi\'nski}
\address{Institute of Mathematics, Jagiellonian University, \L ojasiewicza 6, 30-348 Krak\'ow, Poland}
\email{{lukasz.kosinski}@gazeta.pl}
\thanks{The work is partially supported by the grant of the Polish Ministry for Science and Higher Education No. N N201 361436.}
\begin{document}

\begin{abstract} It is shown that the conjecture on the description of the group of automorphism of the spectral ball posed by Ransford and White is false.
\end{abstract}

\maketitle

Let $\Omega$ denote the spectral ball, that is a domain composed of $n\times n$ complex matrices whose spectral radius is less then $1.$ The natural question that arises in the study of the geometry of $\Omega$ is to classify its group of automorphisms. It is seen that among them there are the following three forms:
\begin{enumerate}[(i)]
\item {\it Transposition}: $\mathcal T: x\mapsto x^t,$
\item {\it M\"obius maps}: $\mathcal M_{\alpha,\gamma}: x\mapsto \gamma(x-\alpha)(1-\overline{\alpha}x)^{-1},$ where $\alpha$ lies in the unit disc and $|\gamma|=1,$
\item {\it Conjugations:} $\mathcal J_u:x\mapsto u(x)^{-1}xu(x),$ where $u:\Omega\to M_n^{-1}$ is a conjugate invariant holomorphic map, i.e. $u(q^{-1}xq)=u(x)$ for each $x\in \Omega$ and $q\in \mathcal M_n^{-1}.$
\end{enumerate}

Ransford and White have asked in \cite{Ran-Whi} whether the compositions of the three above forms generate the whole $\aut (\Omega)$ - the group of automorphisms of the spectral ball. The problems related to this question were considered among others in \cite{Bar-Ran}, \cite{Pas}, \cite{Ros} and \cite{Zwo}.

The aim of this short note is to show that the answer to the conjecture posed by Ransford and White is negative. More precisely, we will present an example of an automorphism of the spectral ball which is not generated by mappings (i), (ii) and (iii).

To simplify the notation we focus our attention on the case $n=2$. The similar examples working for $n\geq 2$ are self-evident.

\medskip

For the convenience of the reader we start with recalling some basic properties of the mappings (i), (ii) and (iii) which will be useful in the sequel. Let $f$ be an automorphism of the spectral ball generated by (i), (ii) and (iii) and preserving the origin. Observe that $\mathcal J_u\circ \mathcal M_{\alpha,\gamma}=\mathcal M_{\alpha, \gamma}\circ \mathcal J_{u\circ \mathcal M_{\alpha,\gamma}}$ for any $u:\Omega\to \mathcal M_n^{-1}$, $|\alpha|<1$ and $|\gamma|=1.$ Therefore, the automorphism $f$ is generated only by the mappings of the form (i) and (iii). Moreover, $f'(0)$ is a linear automorphism of the spectral ball and it follows from \cite{Ran-Whi} that $f'(0).x =\gamma mxm^{-1},$ $x\in \Omega,$ or $f'(0).x=\gamma mx^{t}m^{-1},$ $x\in \Omega,$ for some unimodular constant $\gamma$ and invertible matrix $m.$ Note also that if the first possibility holds, i.e. $f'(0).x=\gamma mxm^{-1}$, then $\gamma^{-1}f$ is of the form (iii).

\begin{theorem}[Example 1]\label{ex1} Let $\varphi$ be an entire holomorphic function on $\mathbb C.$ Put $$f((x_{ij}))=\left(\begin{array}{cc} x_{11} & e^{-\varphi(x_{12}x_{21})}x_{12} \\ e^{\varphi (x_{12}x_{21})}x_{21} & x_{22} \\\end{array} \right),\quad x\in \Omega.$$ Then $f$ is an automorphism of $\Omega$ which is not generated by mappings of the form (i), (ii) and (iii), providing that $\varphi$ is non-constant.
\end{theorem}

\begin{proof}It is clear that $f$ preserves the spectrum, so $f(\Omega)\subset \Omega.$ Thus, it suffices to observe that the inverse to $f$ is given by the formula: $$f^{(-1)}((x_{ij}))=\left(                                                                                           \begin{array}{cc} x_{11} & e^{\varphi(x_{12}x_{21})}x_{12} \\ e^{-\varphi(x_{12}x_{21})}x_{21} & x_{22} \\\end{array} \right),\quad x\in \Omega.$$

What remains to do is to show that $f$ is not generated by the mappings of the form (i), (ii) and (iii).

It is clear that $f'(0).x=mxm^{-1},$ $x\in \Omega,$ for some diagonal $m\in \mathcal M^{-1}_2,$ so we aim at proving that $f$ is not a conjugation $\mathcal J_v$ with $v$ depending only on the conjugacy class of $x$.

It is enough to observe that $f$ restricted to the fiber $\mathfrak S:=\{x\in \Omega:\ \det x=0,\ \tr x=1/2\}$ is not of the form $f(x)=nxn^{-1},$ $x\in \Omega$, for any invertible $n$.
To do it compute $$f\left(\left(                                                                                           \begin{array}{cc} \lambda & \lambda \\ 1/2-\lambda & 1/2-\lambda \\\end{array} \right)\right)=\left(                                                                                           \begin{array}{cc} \lambda & e^{-\varphi(\lambda(1/2-\lambda))}\lambda \\ e^{\varphi(\lambda(1/2-\lambda))}(1/2-\lambda) & 1/2-\lambda, \\\end{array} \right),$$ $\lambda\in \mathbb C,$ and use the trivial fact that for any $\alpha,\beta\in \mathbb C$ the mappings $e^{\varphi(\lambda(1/2-\lambda))}\lambda$ and $\alpha \lambda + \beta$, $\lambda\in \mathbb C$ are not identically equal, provided that $\varphi$ is non-constant (compare for example the singularity at $\infty$).
\end{proof}

\begin{remark}
One can show that the conjugation \begin{equation}\label{conj} \mathcal J(x)=\diag(a(x),a^{-1}(x)) x \diag(a(x),a^{-1}(x))^{-1},\end{equation} where $a$ is a non-vanishing holomorphic function on $\Omega$, is an automorphism of the spectral ball if and only if $a$ depends only on $x_{11},x_{22}$ and $x_{12}x_{21}.$

This remark describes more precisely the situation occurring in the Example~\ref{ex1}. Actually, the automorphism appearing in Example~\ref{ex1} is of the form (\ref{conj}) with $a$ depending only on $x_{12}x_{21}$ (more precisely, $a(x)= e^{-\varphi(x_{12}x_{21})}$).
\end{remark}

\begin{example} Put $u\left(\left(\begin{array}{cc} x_{11} & x_{12} \\ x_{21} & x_{22} \\\end{array} \right)\right):=\left(\begin{array}{cc} 1 & 0 \\ a(x_{12}) & 1 \\\end{array} \right)$, $x\in \Omega$, where $a$ is an entire holomorphic function. Then the conjugation $f(x)=u(x)xu(x)^{-1},$ $x\in \Omega,$ is an automorphism of $\Omega$ non-generated by the automorphisms of the form (i), (ii) and (iii).
\end{example}
\begin{proof} Direct computations allow us to observe that $f^{(-1)}$ is given by the formula $f^{(-1)}(x)=u^{-1}(x) x u(x),$ $x\in \Omega.$ The second part may be proved similarly as in the previous example.
\end{proof}

\begin{remark} We would like to point out that the situation occurring in the above example may be described more precisely. Namely, put $u:=\left(\begin{array}{cc}  1 & 0 \\ a & 1 \\\end{array} \right),$ where $a$ is an arbitrary holomorphic function on $\Omega.$ Then the conjugation $\mathcal J_u: x\mapsto u(x) x u(x)^{-1}$ is an automorphism of $\Omega$ if and only if $a$ depends on $x_{12},$ $\tr x$ and $\det x.$
\end{remark}

\medskip

The problem of describing the group of automorphisms of spectral ball was divided by Rostand (\cite{Ros}) into two parts. The first one is to prove that any normalized automorphism (i.e. an automorphism fixing the origin and satysfying $f'(0)=\id$) is a conjugation $\mathcal J_u$, where $u:\Omega\to \mathcal M_n^{-1}$ is a holomorphic mapping. The second one is to show that the automorphism being a conjugation is of the form $\mathcal J_u$ with $u$ satisfying the condition $u(q^{-1}x q)=u(x)$ for each $x\in \Omega$ and each invertible $q.$ Note that it was shown in this paper that the second statement is false. On the other hand it seems that the answer to the first question is affirmative. So we have two following problems:
\begin{itemize}
\item Is every normalized automorphism of the spectral ball of the form $x\to u(x)xu(x)^{-1}$, where $u:\Omega\to \mathcal M_n^{-1}$ is a holomorphic mapping?
\item Describe all holomorphic $u:\Omega\to \mathcal M_{n}^{-1}$ such that the mapping $x\mapsto u(x)xu(x)^{-1}$ is an automorphism of $\Omega$.
\end{itemize}

\end{document}